\documentclass{amsart}
\usepackage{hyperref}

\newcommand*{\mailto}[1]{\href{mailto:#1}{\nolinkurl{#1}}}

\newtheorem{theorem}{Theorem}[section]
\newtheorem{definition}[theorem]{Definition}
\newtheorem{lemma}[theorem]{Lemma}
\newtheorem{corollary}[theorem]{Corollary}
\newtheorem{hypothesis}[theorem]{Hypothesis}
\newtheorem{remark}[theorem]{Remark}

\newcommand{\R}{\mathbb{R}}
\newcommand{\Z}{\mathbb{Z}}
\newcommand{\N}{\mathbb{N}}
\newcommand{\C}{\mathbb{C}}

\newcommand{\nn}{\nonumber}
\newcommand{\be}{\begin{equation}}
\newcommand{\ee}{\end{equation}}
\newcommand{\bea}{\begin{eqnarray}}
\newcommand{\eea}{\end{eqnarray}}
\newcommand{\ba}{\begin{array}}
\newcommand{\ea}{\end{array}}

\newcommand{\ol}{\overline}
\newcommand{\ul}{\underline}
\newcommand{\ti}{\tilde}

\renewcommand{\mod}{\,\operatorname{mod}\, }
\newcommand{\id}{{\rm 1\hspace{-0.6ex}l}}
\newcommand{\I}{\mathrm{i}}
\newcommand{\E}{\mathrm{e}}

\newcommand{\Ran}{\operatorname{Ran}}
\newcommand{\db}{\mathfrak{D}}

\newcommand{\spr}[2]{\langle #1 , #2 \rangle}
\newcommand{\floor}[1]{\lfloor#1 \rfloor}
\newcommand{\ceil}[1]{\lceil#1 \rceil}

\newcommand{\eps}{\varepsilon}
\newcommand{\vphi}{\varphi}
\newcommand{\sig}{\sigma}
\newcommand{\lam}{\lambda}

\numberwithin{equation}{section}

\begin{document}

\title{Relative Oscillation Theory for Dirac Operators}

\author[R. Stadler]{Robert Stadler}
\address{Faculty of Mathematics\\ University of Vienna\\
Nordbergstrasse 15\\ 1090 Wien\\ Austria}

\author[G. Teschl]{Gerald Teschl}
\address{Faculty of Mathematics\\
Nordbergstrasse 15\\ 1090 Wien\\ Austria\\ and International
Erwin Schr\"odinger
Institute for Mathematical Physics\\ Boltzmanngasse 9\\ 1090 Wien\\ Austria}
\email{\mailto{Gerald.Teschl@univie.ac.at}}
\urladdr{\url{http://www.mat.univie.ac.at/~gerald/}}

\thanks{{\it Research supported by the Austrian Science Fund (FWF) under Grant No.\ Y330}}
\thanks{{\it J. Math. Anal. Appl. {\bf 371}, 638--648 (2010)}}


\begin{abstract}
We develop relative oscillation theory for one-dimensional Dirac operators which, rather than measuring the spectrum of
one single operator, measures the difference between the spectra of two different operators.
This is done by replacing zeros of solutions of one operator by weighted zeros of Wronskians of solutions of two different
operators. In particular, we show that a Sturm-type comparison theorem still holds in this situation and demonstrate how
this can be used to investigate the number of eigenvalues in essential spectral gaps. Furthermore, the connection with
Krein's spectral shift function is established. As an application we extend a result by K.M.\ Schmidt on the finiteness/infiniteness
of the number of eigenvalues in essential spectral gaps of perturbed periodic Dirac operators.
\end{abstract}

\maketitle

\section{Introduction}

To set the stage, let $I=(a,b) \subseteq \R$ (with $-\infty \le
a < b \le \infty$) be an arbitrary interval and consider the
Dirac differential expression
\begin{equation}\label{dirac}
\tau = \frac{1}{\I} \sig_2 \frac{d}{dx} + \phi(x).
\end{equation}
Here
\begin{equation}
\phi(x) =  \phi_{\rm el}(x)\id  + \phi_{\rm am}(x)\sig_1 +
(m+ \phi_{\rm sc}(x)) \sig_3,
\end{equation}
$\sig_1$, $\sig_2$, $\sig_3$ denote the Pauli matrices
\begin{equation}
\sig_1=\left(\ba{cc} 0 & 1 \\ 1 & 0\ea\right), \quad 
\sig_2=\left(\ba{cc} 0 & -\I \\ \I & 0\ea\right), \quad 
\sig_3=\left(\ba{cc} 1 & 0 \\ 0 & -1\ea\right),
\end{equation}
and $m$, $\phi_{\rm sc}$, $\phi_{\rm el}$, and $\phi_{\rm am}$ are interpreted as
mass, scalar potential, electrostatic potential, and anomalous
magnetic moment, respectively (see \cite{th}, Chapter~4). 
As usual we require $m\in[0,\infty)$ and $\phi_{\rm sc}, \phi_{\rm el}, \phi_{\rm
am}
\in L^1_{loc}(I)$ real-valued. We don't include a
magnetic moment $\hat{\tau} = \tau +\sig_2 \phi_{\rm mg}(x)$
since it can be easily eliminated by a simple gauge transformation
$\tau = U \hat{\tau} U^{-1}$, $U =\exp(\I\int^x \phi_{\rm mg}(r) dr)$
(there is also a gauge transformation which gets rid of $\phi_{am}$ or
$\phi_{\rm el}$ (see \cite{ls}, Section~7.1.1)). 

If $\tau$ is limit point at both $a$ and $b$, then $\tau$ gives rise to a unique
self-adjoint operator $H$ when defined maximally (cf., e.g., \cite{ls},
\cite{wdl}, \cite{wd1}). Otherwise, we need to fix a boundary
condition at each endpoint where $\tau$ is limit circle.

Explicitly, $H$ is given by
\begin{equation}
H: \ba[t]{lcl} \db(H) &\to& L^2(I,\C^2) \\ f &\mapsto& \tau f \ea ,
\end{equation}
where
\begin{equation} \label{domH}
\db(H) = \{ f \in L^2(I,\C^2) | \ba[t]{l} f \in AC_{loc}(I,\C^2), \, \tau f \in
L^2(I,\C^2),\\ W_a(u_-,f) = W_b(u_+,f) =0 \} \ea
\end{equation}
with
\begin{equation}
W_x(f,g) = \I \spr{f^*(x)}{\sig_2 g(x)} = f_1(x) g_2(x) - f_2(x) g_1(x)
\end{equation}
the usual Wronskian (we remark that the limit $W_{a,b}(.,..) = \lim_{x \to
a,b} W_x(.,..)$ exists for functions as in (\ref{domH})). Here the function
$u_-$ (resp.\ $u_+$) used to generate the boundary condition at $a$
(resp.\ $b$) can be chosen to be a nontrivial solution of $\tau u =0$ if $\tau$
is limit circle at $a$ (resp.\ $b$) and zero else.

We refer to the monographs \cite{ls}, \cite{wdl}, \cite{wei2} for background
and also \cite{th} for further information about Dirac operators and their
applications.

However, even though the Dirac operator is as important to
relativistic quantum mechanics as the Schr\"odinger operator to
nonrelativistic quantum mechanics, much less is known about its discrete
spectrum. The main reason of course being that in contradistinction to
typical Schr\"odinger operators, Dirac operators are not bounded from below
and thus approaches relying on semi-boundedness are not applicable.

Our aim in the present paper is to develop what we will call relative oscillation
theory for a pair of Dirac operators $H_1$ and $H_0$ associated with
two potentials $\phi_1$ and $\phi_0$ as above. As we will show, it turns
out to be an effective tool for both counting eigenvalues in essential spectral
gaps as well as for investigation the accumulation of eigenvalues at the
boundary of an essential spectral gap.

Let $\spr{f}{g} = f_1^* g_1 + f_2^* g_2$ and $|f| = \sqrt{|f_1|^2+|f_2|^2}$ denote the
scalar product and norm in $\C^2$. Our key ingredient will be the Wronskian
of two (nontrivial) real-valued solutions $u_0$ and $u_1$ satisfying
$\tau_0 u_0 = \lam_0 u_0$ and $\tau_1 u_1 = \lam_1 u_1$. Then
we define a Pr\"ufer angle for the Wronskian $W(u_0,u_1)$ via
\begin{equation}
\begin{pmatrix} W_x(u_1,u_0) \\ W_x(u_1,-\I\sig_2 u_0) \end{pmatrix} =
R(x) \begin{pmatrix} \sin(\psi(x)) \\ \cos(\psi(x)) \end{pmatrix}.
\end{equation}
Note that $\psi(x)$ is uniquely determined up to a multiple of $2\pi$ by
the requirement that $\psi(x)$ should be continuous since the two
Wronskians cannot vanish simultaneously.

The total difference
\begin{equation}
\#_{(c,d)}(u_0,u_1) = \ceil{\Delta_{1,0}(d) / \pi} - \floor{\Delta_{1,0}(c) / \pi} -1
\end{equation}
will then be called the weighted number of sign flips of the
Wronskian $W(u_0,u_1)$ in the interval $(c,d) \subset I$ (with $a<c<d<b$).
Here $\floor{x} = \max \{n \in \Z \,|\, n \leq x \}$ and $\ceil{x} = \min \{n \in \Z \,|\, n \geq x \}$
are the usual floor and ceiling functions.

In fact, $\#_{(c,d)}(u_0,u_1)$ counts the number of sign flips of $W(u_0,u_1)$ where
a sign flip is counted as $+1$ if $\psi$ increases along the sign flip and as  $-1$
if $\psi$ decreases. Moreover, one can show that a zero $x_0$ is counted as $+1$ if
$\spr{u_0(x_0)}{\Delta\phi(x_0)u_0(x_0)} >0$ and as $-1$ if
$\spr{u_0(x_0)}{\Delta\phi(x_0)u_0(x_0)} <0$, where
\begin{equation}
\Delta\phi= \phi_1 -\phi_0.
\end{equation}

We will also set
\begin{equation}
\#(u_0,u_1) = \lim_{c\downarrow a, d\uparrow b} \#_{(c,d)}(u_0,u_1)
\end{equation}
provided this limit exists. This will for example be the case if the perturbation is
of a definite sign, $\Delta\phi(x) \ge 0$ or $\Delta\phi(x) \le 0$, at least for $x$
near $a$ and $b$. We will call $\tau_1-\lam_1$ relatively nonoscillatory with respect to
$\tau_0-\lam_0$ if $\#(u_0,u_1)$ is finite and relatively oscillatory otherwise.

Our first result implies that if we choose $u_0$ and $u_1$ to be Weyl solutions, then
the weighted number of sign flips counts precisely the eigenvalue difference.
Recall that a solution $u_-(z,.)$ of $\tau u =z u$ is called Weyl solution at $a$ if it
is square integrable near $a$ and fulfills the boundary condition of $H$ at $a$
(if there is any, i.e., if $\tau$ is limit circle at $a$). Such a solution is unique
up to a constant if it exists (e.g.\ if $z\not\in \sig_{ess}(H)$) and it can be chosen to
be real for $z\in\R$. Similarly a Weyl solution $u_+(z,.)$ at $b$ is defined.

Finally, denote by $P_\Omega(H)$, $\Omega\subseteq\R$, the family of spectral
projections associated with the self-adjoint operator $H$ (see e.g.\ \cite{temmq}).

\begin{theorem}\label{thm:wronskzeros}
Let $H_0$, $H_1$ be self-adjoint operators associated with $\tau_0$, $\tau_1$,
respectively, and separated boundary conditions. Suppose 
\begin{enumerate}
\item
$\Delta\phi \le 0$, near singular endpoints,
\item
$\lim_{x\to a} \Delta\phi(x) = 0$ if $a$ is singular and
$\lim_{x\to b} \Delta\phi(x) = 0$ if $b$ is singular,
\item
$H_0$ and $H_1$ are associated with the same boundary conditions near $a$ and $b$,
that is, $u_{0,-}(\lam)$ satisfies the boundary condition of $H_1$ at $a$ (if any) and
$u_{1,+}(\lam)$ satisfies the boundary condition of $H_0$ at $b$ (if any).
\end{enumerate}

Suppose $\lam_0 < \inf\sig_{ess}(H_0)$. Then
\be\label{eq:wronskzeros0}
\dim\Ran P_{(-\infty,\lam_0)}(H_1) - \dim\Ran P_{(-\infty,\lam_0]}(H_0)
= \#(u_{1,\mp} (\lam_0), u_{0,\pm} (\lam_0)).
\ee

Suppose $\sig_{ess}(H_0) \cap [\lam_0,\lam_1] = \emptyset$. Then $\tau_1-\lam_0$ is
relatively nonoscillatory with respect to $\tau_0-\lam_0$ and
\begin{align}
\nonumber
&\dim\Ran P_{[\lam_0, \lam_1)} (H_1) - \dim\Ran P_{(\lam_0, \lam_1]} (H_0)\\ \label{eq:wronskzeros1}
& \qquad = \#(u_{1,\mp} (\lam_1), u_{0,\pm} (\lam_1)) - 
\#(u_{1,\mp} (\lam_0), u_{0,\pm} (\lam_0)).
\end{align}
\end{theorem}

The proof will be given at the end of Section~\ref{sec:rot}.

\begin{remark}
Note that condition (ii) implies $\sig_{ess}(H_0)=\sig_{ess}(H_1)$ (cf.\ Lemma~\ref{lem:nonoscingap}
below). In addition, (ii) implies that any function which is in $\db(\tau_0)$ near $a$ (or $b$)
is also in $\db(\tau_1)$ near $a$ (or $b$), and vice versa. Hence condition (iii) is
well-posed.
\end{remark}

In the case where the resolvent difference of $H_1$ and $H_0$ is trace class, the difference
in \eqref{eq:wronskzeros1} as opposed to \eqref{eq:wronskzeros0} can be avoided if we
replace the left-hand side by Krein's spectral shift function $\xi(\lam,H_1,H_0)$ (see \cite{yafams}
for more information on Krein's spectral shift function). In order to fix the unknown constant in the spectral
shift function, we will require that $H_0$ and $H_1$ are connected via a path within the set of
operators  whose resolvent difference with $H_0$ are trace class. Hence we will require

\begin{hypothesis} \label{hyp:h0h1}
Suppose $H_0$  and $H_1$ are self-adjoint operators associated with $\tau_0$ and $\tau_1$
and separated boundary conditions. Assume that
\begin{itemize}
\item $\Delta\phi$ is relatively bounded with respect to $H_0$ with $H_0$-bound less than one, and
\item $\sqrt{|\Delta\phi|} (H_0-z)^{-1}$ is Hilbert--Schmidt for one (and hence for all)
$z\in\rho(H_0)$.
\end{itemize}
\end{hypothesis}

It was shown in \cite[Sect.~8]{kt} that these conditions ensure that we can interpolate
between $H_0$ and $H_1$ using operators $H_\eps$, $\eps\in[0,1]$, such that the resolvent difference
of $H_0$ and $H_\eps$ is continuous in $\eps$ with respect to the trace norm. Hence we can fix
$\xi(\lam,H_1,H_0)$ by requiring $\eps \mapsto \xi(\lam,H_\eps,H_0)$ to be continuous in
$L^1(\R,(\lam^2+1)^{-1}d\lam)$, where we of course set $\xi(\lam,H_0,H_0)=0$.
While $\xi$ is only defined a.e., it is constant on the intersection
of the resolvent sets $\R\cap\rho(H_0)\cap\rho(H_1)$, and we will require it to be continuous there.
In particular, note that by Weyl's theorem the essential spectra of $H_0$ and
$H_1$ are equal, $\sig_{ess}(H_0)=\sig_{ess}(H_1)$. Then we have the following result:

\begin{theorem} \label{thmsing}
Let $H_0$, $H_1$ satisfy Hypothesis~\ref{hyp:h0h1}. Then
for every $\lam\in\R\cap\rho(H_0)\cap\rho(H_1)$ we have
\be
\xi(\lam,H_1,H_0) =
\#(\psi_{0,\pm}(\lam), \psi_{1,\mp}(\lam)).
\ee
\end{theorem}

Again, the proof will be given at the end of Section~\ref{sec:rot}.

In particular, this result implies that under these assumptions $\tau_1-\lam$ is relatively nonoscillatory
with respect to $\tau_0-\lam$ for every $\lam$ in an essential spectral gap.

Concerning the history of these results we mention that the analogs of Theorem~\ref{thm:wronskzeros}
and Theorem~\ref{thmsing} were first given in the case of Sturm--Liouville operators by Kr\"uger and
Teschl \cite{kt}, \cite{kt2} extending earlier work of Gesztesy, Simon, and Teschl \cite{gst} which
corresponded to the case $H_1=H_0$. In the case of Dirac operators the case $H_1=H_0$ was
first given in Teschl \cite{toscd}.

Finally, we will show how $\#(u_0,u_1)$ can be used to settle the question whether the eigenvalues introduced
by a given perturbation will accumulate at a boundary point of the essential spectrum and apply this
to the case of perturbed periodic Dirac operators.

We first recall some basic facts from the theory of periodic Dirac operators
(cf., e.g., \cite{wdl}, Chapter~12, \cite{wei2}, Chapter~16). Let $H_0$ be a Dirac
operator associated with periodic potential $\phi_0$ of period $\alpha>0$, that is, $\phi_0(x+\alpha) =
\phi_0(x)$, $x\in I=(a,\infty)$. The essential spectrum of $H_0$ is purely absolutely
continuous and consists of a countable number of bands, that is,
\begin{equation}
\sig_{ess}(H_0) = \bigcup_{j\in\Z} [E_{2j},E_{2j+1}]
\end{equation}
with $\cdots E_{2j} < E_{2j+1} \le E_{2j+2} < E_{2j+3}\cdots$. In addition,
in every essential spectral gap there can be at most one eigenvalue.

Moreover, Floquet theory implies the existence of an (anti-)periodic solution
$u_0(E_j,x)$ at each boundary point of the essential spectrum.

To phrase our result, we recall the iterated logarithm $\log_n(x)$ which is defined recursively via
\[
\log_0(x) = x, \qquad \log_n(x) = \log(\log_{n-1}(x)).
\]
Here we use the convention $\log(x)=\log|x|$ for negative values of $x$. Then
$\log_n(x)$ will be continuous for $x>\E_{n-1}$ and positive for $x>\E_n$, where
$\E_{-1}=-\infty$ and $\E_n=\E^{\E_{n-1}}$. Abbreviate further
\[
L_n(x) = \frac{1}{\log_{n+1}'(x)} = \prod_{j=0}^n \log_j(x).
\]
Explicitly we have
\[
L_0(x) =x, \quad L_1(x) = x \log(x), \quad L_2(x) = x \log(x)\log(\log(x)), \quad \dots
\]
With this notation we have the following result:

\begin{theorem}\label{thm:periodic}
Let $E_j$ be a boundary point of the essential spectrum of the periodic operator
$H_0$ and let $u_0(x)$ be a corresponding (anti-)periodic solution of $\tau_0 u_0 = E_j u_0$.

Suppose
\begin{equation}\label{asumper}
\phi_1(x) = \phi_0(x) -\frac{1}{4} \sum_{k=0}^n \frac{1}{L_k(x)^2} \phi_{1,k} + o(L_n(x)^{-2})
\end{equation}
for some constant matrices $\phi_{1,k}$, $0\le k \le n$, and define
\begin{align}\nn
A &= \frac{2}{\alpha} \int_0^\alpha \frac{\spr{u(x)}{((m + \phi_{0,\rm sc}(x)) \sig_3 + \phi_{0,\rm am}(x) \sig_1) u(x)}}{|u(x)|^4} dx,\\
B_k &= -\frac{1}{\alpha} \int_0^\alpha \spr{u(x)}{\phi_{1,k} u(x)} dx, \quad 0\le k \le n.
\end{align}
Then the eigenvalues of $H_1$ accumulate at $E_j$ if
\begin{equation}
A B_0 = \cdots = A B_{n-1} = 1 \quad\text{and}\quad  A B_n > 1
\end{equation}
and the do not  accumulate at $E_j$ if
\begin{equation}
A B_0 = \cdots = A B_{n-1} = 1 \quad\text{and}\quad A B_n < 1.
\end{equation}
\end{theorem}

The proof will be given at the end of Section~\ref{sec:roc}.

In the case of Sturm--Liouville operators this result originates in the work of Rofe-Beketov \cite{rb2}--\cite{rb5}
(see also the recent monograph \cite{rbk}) who proved the case $n=0$. His work was recently
improved by Schmidt \cite{kms} who gave a new proof and obtained the cases $n=0,1$. Extending
the approach by Schmidt the general case was obtained in Kr\"uger and Teschl \cite{kt3}. Schmidt
also established the case $n=0,1$ for Dirac operators in \cite{kms2}. In his paper \cite{kms2} he
also gives an equivalent formulation for the criterion in terms of the gradient of the Floquet
discriminant and shows how the above criterion can be applied to radial Dirac operators via
a transformation from \cite{kms0}. In fact, if
\be
\tau_k = \frac{1}{\I} \sig_2 \frac{d}{dr} + \frac{k}{r} \sig_3 + \phi(r), \qquad r\in(0,\infty),
\ee
is a radial Dirac operator (i.e.\  one which arises by separation of variables in spherical coordinates \cite[Sect.~4.6.6]{th}),
then the unitary transformation (\cite[Lem.~3]{kms0})
\be
U f(r) = \begin{pmatrix} \cos(\theta(r)) & - \sin(\vartheta(r))\\ \cos(\vartheta(r)) & \sin(\vartheta(r)) \end{pmatrix}
\begin{pmatrix} f_1(r) \\ f_2(r) \end{pmatrix}, \quad
\vartheta(r) = \frac{1}{2} \arctan\Big(\frac{k}{r}\Big),
\ee
transforms $\tau$ to
\be
U^* \tau U = \frac{1}{\I} \sig_2 \frac{d}{dr} + \left(\sqrt{1+\frac{k^2}{r^2}} -1\right) \sig_3 
+\frac{k}{2(r^2+k^2)} \id + \phi(r).
\ee
Since
\be
\left(\sqrt{1+\frac{k^2}{r^2}} -1\right) \sig_3 
+\frac{k}{2(r^2+k^2)} \id = \frac{k}{2} (k \sig_3 + \id) \frac{1}{r^2} + O(r^{-4})
\ee
our result is directly applicable to this situation.

We also refer to \cite{kms2} and the recent work by Cojuhari \cite{co} for more
on the history of this problem and references to related results. Analogous results for the discrete
case, Jacobi matrices, can be found in \cite{at}.

\section{Relative Oscillation Theory}
\label{sec:rot}

After these preparations we are now ready to develop relative oscillation theory. Our
presentation closely follows \cite{kt}.

\begin{definition}\label{def:wsf}
For $\tau_0$, $\tau_1$ possibly singular Dirac operators as in (\ref{dirac}) on $(a,b)$,
we define
\be
\underline{\#}(u_0,u_1) = \liminf_{d \uparrow b,\,c \downarrow a} \#_{(c,d)}(u_0,u_1)
\quad\mbox{and}\quad
\overline{\#}(u_0,u_1) = \limsup_{d \uparrow b,\,c \downarrow a} \#_{(c,d)}(u_0,u_1),
\ee
where $\tau_j u_j = \lam_j u_j$, $j=0,1$.

We say that $\#(u_0,u_1)$ exists, if
$\overline{\#}(u_0,u_1)=\underline{\#}(u_0,u_1)$, and write
\be
\#(u_0,u_1) = \overline{\#}(u_0,u_1)=\underline{\#}(u_0,u_1)
\ee
in this case.
\end{definition}

By Lemma~\ref{lem:delinc} below one infers that $\#(u_0,u_1)$ exists
if $\phi_0-\lam_0- \phi_1+\lam_1$ has the same definite sign near the endpoints $a$ and $b$.
On the other hand, note that $\#(u_0,u_1)$ might not exist
even if both $a$ and $b$ are regular, since the difference of Pr\"ufer angles might
oscillate around a multiple of $\pi$ near an endpoint. Furthermore, even if it exists,
one has $\#(u_0,u_1) = \#_{(a,b)}(u_0,u_1)$ only if there are no zeros at the
endpoints (or if  $\phi_0-\lam_0-\phi_1+\lam_1\ge 0$ at least near the endpoints).

We begin with our analog of Sturm's comparison theorem for
zeros of Wronskians. We will also establish a triangle-type inequality
which will help us to provide streamlined proofs below.
Both results follow as in \cite{kt}.

\begin{theorem}[Comparison theorem for Wronskians]\label{thm:scw}
Suppose $u_j$ satisfies $\tau_j u_j = \lam_j u_j$, $j=0,1,2$,
where $\lam_0 -\phi_0 \le \lam_1 - \phi_1 \le \lam_2 - \phi_2$.

If $c<d$ are two zeros of $W_x(u_0,u_1)$ such that $W_x(u_0,u_1)$ does not
vanish identically, then there is at least one sign flip of $W_x(u_0,u_2)$ in $(c,d)$.
Similarly, if $c<d$ are two zeros of $W_x(u_1,u_2)$ such that $W_x(u_1,u_2)$ does not
vanish identically, then there is at least one sign flip of $W_x(u_0,u_2)$ in $(c,d)$.
\end{theorem}

\begin{theorem}[Triangle inequality for Wronskians]\label{thm:wtriang}
Suppose $u_j$, $j=0,1,2$ are given real-valued non-vanishing vector functions. Then
\be 
\ul{\#}(u_0,u_1) + \ul{\#}(u_1,u_2) - 1 \leq \ul{\#}(u_0,u_2) \leq
\ul{\#}(u_0,u_1) + \ul{\#}(u_1,u_2) + 1,
\ee
and similarly for $\ul{\#}$ replaced by $\ol{\#}$.
\end{theorem}

\begin{definition}\label{def:relosc}
We call $\tau_1$ relatively nonoscillatory with respect to
$\tau_0$, if the quantities $\underline{\#}(u_0, u_1)$ and
$\overline{\#}(u_0, u_1)$ are finite for all solutions
$\tau_j u_j = 0$, $j = 0,1$.

We call $\tau_1$ relatively oscillatory with respect to
$\tau_0$, if one of the quantities $\underline{\#}(u_0, u_1)$ or
$\overline{\#}(u_0, u_1)$ is infinite for some solutions
$\tau_j u_j = 0$, $j = 0,1$.
\end{definition}

Note that this definition is in fact independent of the solutions
chosen as a straightforward application of our triangle inequality
(cf.\ Theorem~\ref{thm:wtriang}) shows.

\begin{corollary}\label{cor:indisol}
Let $\tau_j u_j = \tau_j v_j = 0$, $j=0,1$. Then
\be
|\underline{\#}(u_0, u_1) - \underline{\#}(v_0, v_1)| \le 4,\quad
|\overline{\#}(u_0, u_1) - \overline{\#}(v_0, v_1)| \le 4.
\ee
\end{corollary}

The bounds can be improved using our comparison theorem for
Wronskians to be $\leq 2$ in the case of perturbations of definite sign.

To demonstrate the usefulness of Definition~\ref{def:relosc},
we now establish its connection with the spectra
of self-adjoint operators associated with $\tau_j$, $j=0,1$.

\begin{theorem} \label{thm:rosc}
Let $H_j$ be self-adjoint Dirac operators associated with $\tau_j$, $j=0,1$. Then
\begin{enumerate}
\item
$\tau_0-\lam_0$ is relatively nonoscillatory with respect to $\tau_0-\lam_1$
if and only if $\dim\Ran P_{(\lam_0,\lam_1)}(H_0)<\infty$.
\item
Suppose $\dim\Ran P_{(\lam_0,\lam_1)}(H_0)<\infty$ and
$\tau_1-\lam$ is relatively nonoscillatory with respect to $\tau_0-\lam$ for one
$\lam \in [\lam_0,\lam_1]$.  Then it is relatively nonoscillatory for
all $\lam\in[\lam_0,\lam_1]$ if and only if $\dim\Ran P_{(\lam_0,\lam_1)}(H_1)<\infty$.
\end{enumerate}
\end{theorem}

\begin{proof}
Item (i) is \cite[Thm.~4.5]{toscd} and item (ii) follows as in \cite{kt}.
\end{proof}

For a practical application of this theorem one needs of course criteria
when  $\tau_1-\lam$ is relatively nonoscillatory with respect to $\tau_0-\lam$
for $\lam$ inside an essential spectral gap.

\begin{lemma}\label{lem:nonoscingap}
Let $\lim_{x\to a} (\phi_0(x) - \phi_1(x)) = 0$ if $a$ is singular, and
similarly, $\lim_{x\to b} (\phi_0(x) - \phi_1(x)) = 0$ if $b$ is singular.
Then $\sig_{ess}(H_0)=\sig_{ess}(H_1)$ and
$\tau_1 - \lam$ is relatively nonoscillatory with respect
to $\tau_0 - \lam$ for $\lam \in \R \backslash\sigma_{ess}(H_0)$.
\end{lemma}

\begin{proof}
Since $\tau_1$ can be written as $\tau_1 = \tau_0 + \ti{\phi}_0 + \ti{\phi}_1$, where
$\ti{\phi}_0$ has compact support near singular endpoints and $|\ti{\phi}_1|<\eps$,
for arbitrarily small $\eps>0$, we infer that $R_{H_1}(z) - R_{H_0}(z)$ is the norm
limit of compact operators. Thus $R_{H_1}(z) - R_{H_0}(z)$ is compact and hence
$\sig_{ess}(H_0)=\sig_{ess}(H_1)$.

Let $\delta>0$ be the distance of $\lam$ to the essential spectrum and choose
$a < c < d < b$, such that
$$
|\phi_1(x) - \phi_0(x)| \le \delta/2,\qquad x\not\in(c,d).
$$
Clearly $\#_{(c,d)} (u_0,u_1) < \infty$, since both operators are regular
on $(c,d)$. Moreover, observe that
$$
\phi_0 - \lam_+ \leq \phi_1- \lam \leq \phi_0- \lam_-, \qquad
\lam_\pm=\lam\pm\delta/2,
$$
on $I = (a,c)$ or $I =(d,b)$. Then Theorem~\ref{thm:rosc}~(i) implies
$\#_I(u_0(\lam_-), u_0(\lam_+)) < \infty$ and invoking Theorem~\ref{thm:scw}
shows $\#_I(u_0(\lam_\pm), u_1(\lam)) < \infty$.
From Theorem~\ref{thm:wtriang} and \ref{thm:rosc}~(i) we infer
$$
\ol{\#}_I(u_0(\lam),u_1(\lam)) < \#_I(u_0(\lam),u_0(\lam_+)) +
\#_I(u_0(\lam_+), u_1(\lam)) + 1 < \infty,
$$
and similarly for $\ul{\#}_I(u_0(\lam),u_1(\lam))$.
This shows that $\tau_1-\lam$ is relatively nonoscillatory with respect to $\tau_0$.
\end{proof}

Our next task is to reveal the precise relation between the number of
weighted sign flips and the spectra of $H_1$ and $H_0$. The special
case $H_0=H_1$ is covered by

\begin{theorem}[{\cite[Thm.~4.5]{toscd}}] \label{thm:gst}
Let $H_0$ be a self-adjoint operator associated with $\tau_0$ and suppose
$[\lam_0,\lam_1]\cap\sig_{ess}(H_0)=\emptyset$.
Then
\be
\dim\Ran P_{(\lam_0, \lam_1)} (H_0) = \#(\psi_{0,\mp}(\lam_0),\psi_{0,\pm}(\lam_1)).
\ee
\end{theorem}

Combining this result with our triangle inequality already gives some rough
estimates in the spirit of Weidmann \cite{wd1} who treats the case $H_0=H_1$.

\begin{lemma}\label{lem:estproj}
For $j=0,1$ let $H_j$ be a self-adjoint operator associated with $\tau_j$ and separated
boundary conditions. Suppose 
that $(\lam_0, \lam_1)\subseteq\R\backslash(\sig_{ess}(H_0)\cup\sig_{ess}(H_1))$, then
\begin{align}
\nonumber
\dim\Ran P_{(\lam_0, \lam_1)} (H_1) &-
\dim\Ran P_{(\lam_0, \lam_1)} (H_0)\\
& \leq \underline{\#}(\psi_{1,\mp} (\lam_1), \psi_{0,\pm} (\lam_1)) - 
\overline{\#}(\psi_{1,\mp} (\lam_0), \psi_{0,\pm} (\lam_0)) + 2,
\end{align}
respectively,
\begin{align}
\nonumber
\dim\Ran P_{(\lam_0, \lam_1)} (H_1) &-
\dim\Ran P_{(\lam_0, \lam_1)} (H_0)\\
& \geq \overline{\#}(\psi_{1,\mp} (\lam_1), \psi_{0,\pm} (\lam_1)) - 
\underline{\#}(\psi_{1,\mp} (\lam_0), \psi_{0,\pm} (\lam_0)) - 2.
\end{align}
\end{lemma}

Given these preparations the proofs of Theorem~\ref{thm:wronskzeros} and Theorem~\ref{thmsing}.
can be done as in \cite{kt}.

\begin{proof}[Proof of Theorem~\ref{thm:wronskzeros}.]
For the proof one can literally follow the arguments in Section~6 of \cite{kt}. The only noteworthy difference
is that in Lemma~6.4 one has to use the $\limsup$ of the largest eigenvalue and the $\liminf$ of the lowest eigenvalue
of $\ti{\phi}$.
\end{proof}

\begin{proof}[Proof of Theorem~\ref{thmsing}.]
For the proof one can literally follow the arguments in Section~7 of \cite{kt}.
\end{proof}

\section{More on Pr\"ufer angles and the case of regular operators}
\label{sec:proofreg}

The purpose of this section is to collect some further facts on Pr\"ufer angles for
Wronskians and to prove Theorem~\ref{thm:wronskzeros} in the case of regular
operators. Even tough the Pr\"ufer angle $\Delta_{1,0}$ introduced below is
different from $\psi$ used in the introduction it will be equivalent for our purpose
(cf.\ Definition~\ref{def:eqpa} below). We closely follow \cite{kt} and we will provide
proofs only when there is a significant difference to the Sturm--Liouville case.

We first introduce Pr\"ufer variables for $u\in C(I,\R^2)$ defined by
\begin{equation}
u_1(x)=\rho_u(x)\sin(\theta_u(x)) \qquad u_2(x)=\rho_u(x)
\cos(\theta_u(x)).
\end{equation}
If $u$ is never $(0,0)$ and $u$ is continuous, then $\rho_u$ is
positive and $\theta_u$ is uniquely determined once a value of
$\theta_{u}(x_0)$, $x_0\in I$ is chosen by the requirement $\theta_u \in C(I,\R)$.

The connection with the Wronskian is given by
\begin{equation} \label{wpruefer}
W_x(u,v)= -\rho_u(x)\rho_v(x)\sin(\Delta_{v,u}(x)), \qquad
\Delta_{v,u}(x) = \theta_v(x)-\theta_u(x).
\ee
Hence the Wronskian vanishes if and only if the two Pr\"ufer angles differ by
a multiple of $\pi$. We will call the total difference
\be
\#_{(c,d)}(u_0,u_1) = \ceil{\Delta_{1,0}(d) / \pi} - \floor{\Delta_{1,0}(c) / \pi} -1
\ee
the number of weighted sign flips in $(c,d)$, where we have written $\Delta_{1,0}(x)=
\Delta_{u_1,u_0}$ for brevity.
 
Next, let us take two real-valued (nontrivial) solutions $u_j$, $j=1,2$, of $\tau_j u_j =\lam_j u_j$ and
associated Pr\"ufer variables $\rho_j$, $\theta_j$. Since we can
replace $\phi \to \phi - \lam$ it is no restriction to assume $\lam_0=\lam_1=0$.

Under these assumptions $W_x(u_0,u_1)$ is absolutely continuous and satisfies
\be \label{dwr}
W'_x(u_0,u_1) = \spr{u_0(x)}{(\phi_1(x)-\phi_0(x))u_1(x)}.
\ee

\begin{lemma}\label{lem:delinc}
Abbreviate $\Delta_{1,0}(x) = \theta_1(x) - \theta_0(x)$ and suppose
$\Delta_{1,0}(x_0) \equiv 0 \mod \pi$. If $-\spr{u_0(x)}{\Delta\phi(x) u_1(x)}$ is
(i) negative, (ii) zero, or (iii) positive for a.e.\ $x\in(x_0,x_0+\eps)$
respectively for a.e.\ $x\in(x_0-\eps,x_0)$ for some $\eps>0$, then the same is true for
$(\Delta_{1,0}(x) - \Delta_{1,0}(x_0))/(x - x_0)$.
\end{lemma}

Hence $\#_{(c,d)}(u_0,u_1)$ counts the weighted sign flips of the Wronskian
$W_x(u_0,u_1)$, where a sign flip is counted as $+1$ if $-\Delta\phi$ is positive
in a neighborhood of the sign flip, it is counted as $-1$ if $-\Delta\phi$ is negative
in a neighborhood of the sign flip. If $\Delta\phi$ changes sign (i.e., it is positive
on one side and negative on the other) the Wronskian will not change its sign.
In particular, we obtain:

\begin{lemma}\label{lemnbzer2}
Let $u_0$, $u_1$ solve $\tau_j u_j = 0$, $j=0,1$, where $\Delta\phi\le 0$.
Then $\#_{(a,b)}(u_0,u_1)$ equals the number sign flips of $W(u_0,u_1)$
inside the interval $(a,b)$.
\end{lemma}

In the case $\Delta\phi\ge 0$ we get of course the corresponding negative number
except for the fact that zeros at the boundary points are counted as well since
$\floor{-x}= -\ceil{x}$. That is, if $\Delta\phi< 0$, then $\#_{(c,d)}(u_0,u_1)$
equals the number of zeros of the Wronskian in $(c,d)$ while if $\Delta\phi> 0$,
it equals minus the number of zeros in $[c,d]$. In the next theorem we will see that
this is quite natural. In addition, note that $\#(u,u) = -1$.

Finally, we establish the connection with the spectrum of regular operators.
A finite end point is called regular if all entries of $\phi$ are integrable near
this end point.  In this case boundary values for
all functions exist at this end point. In particular, $\tau$ is called regular if
both end points $a,b$ are regular. In the regular case the resolvent of
$H$ is Hilbert-Schmidt and hence the spectrum is purely discrete (i.e.,
$\sig_{ess}(H)=\emptyset$).

\begin{theorem} \label{thm:reg}
Let $H_0$, $H_1$ be regular Sturm--Liouville operators associated with $\tau_0$, $\tau_1$ and
the same boundary conditions at $a$ and $b$. Then
\be \label{eqreg}
\dim\Ran\, P_{(-\infty, \lam_1)}(H_1) - \dim\Ran\, P_{(-\infty, \lam_0]}(H_0) =
\#_{(a,b)}(u_{0,\pm}(\lam_0), u_{1,\mp}(\lam_1)).
\ee
\end{theorem}

The proof will be given below employing interpolation between
$H_0$ and $H_1$, using $H_\eps= (1-\eps) H_0 + \eps H_1$ together with
a careful analysis of Pr\"ufer angles.

It is important to observe that in the special case $H_1=H_0$, the left-hand side
equals $\dim\Ran\, P_{(\lam_1,\lam_0)}(H_0)$ if $\lam_1>\lam_0$ and
$-\dim\Ran\, P_{[\lam_0,\lam_1]}(H_0)$ if $\lam_1<\lam_0$. This is
of course in accordance with our previous observation that
$\#(u_{0,\pm}(\lam_0), u_{1,\mp}(\lam_1))$ equals the
number of zeros in $(a,b)$ if $\lam_1>\lam_0$ while it equals minus
the numbers of zeros in $[a,b]$ if $\lam_1<\lam_0$.

Now let us suppose that $\tau_{0,1}$ are both regular at $a$ and $b$ with boundary conditions
\be \label{bc}
\cos(\alpha) f_1(a) - \sin(\alpha) f_2(a) =0 , \quad
\cos(\beta) f_1(b) - \sin(\beta) f_2(b) =0.
\ee
Hence we can choose $u_\pm(\lam,x)$ such that
$u_-(\lam,a)=(\sin(\alpha),\cos(\alpha))$ respectively
$u_+(\lam,b)= (\sin(\beta),\cos(\beta))$. In particular,
we may choose
\be \label{normtha}
\theta_-(\lam,a) =\alpha \in [0,\pi), \quad -\theta_+(\lam,b) = \pi -\beta \in
[0,\pi).
\ee

Next we introduce
\be
\tau_\eps = \tau_0 + \eps (\phi_1 -\phi_0)
\ee
and investigate the dependence with respect to $\eps\in[0,1]$. If $u_\eps$
solves $\tau_\eps u_\eps = 0$, then the corresponding Pr\"ufer angles satisfy
\be
\dot{\theta}_\eps(x) = -\frac{W_x(u_\eps, \dot{u}_\eps)}{\rho_\eps^2(x)},
\ee
where the dot denotes a derivative with respect to $\eps$.

\begin{lemma} \label{prwpsiepsdot}
We have
\be
W_x(u_{\eps,\pm}, \dot{u}_{\eps,\pm}) = \left\{
\begin{array}{l} \int_x^b \spr{u_{\eps,+}(r)}{(\phi_0(r)-\phi_1(r)) u_{\eps,+}(r)} dr \\ -\int_a^x
\spr{u_{\eps,-}(r)}{(\phi_0(r)-\phi_1(r)) u_{\eps,-}(r)}  dr \end{array} \right. ,
\ee
where the dot denotes a derivative with respect to $\eps$ and
$u_{\eps,\pm}(x)= u_{\eps,\pm}(0,x)$.
\end{lemma}

Denoting the Pr\"ufer angles of $u_{\eps,\pm}(x)= u_{\eps,\pm}(0,x)$ by $\theta_{\eps,+}(x)$,
this result implies for $\phi_0-\phi_1\ge 0$,
\begin{align} \nn
\dot{\theta}_{\eps,+}(x) &= -\frac{\int_x^b \spr{u_{\eps,+}(r)}{(\phi_0(r)-\phi_1(r)) u_{\eps,+}(r)}  dr}{\rho_{\eps,+}(x)^2}
\le 0, \\ \label{thetadot}
\dot{\theta}_{\eps,-}(x) &= \frac{\int_a^x \spr{u_{\eps,-}(r)}{(\phi_0(r)-\phi_1(r)) u_{\eps,-}(r)} dr}{\rho_{\eps,-}(x)^2}
\ge 0,
\end{align}
with strict inequalities if $\phi_0>\phi_1$ on a subset of positive Lebesgue measure of $(x,b)$,
respectively $(a,x)$.

Now we are ready to investigate the associated operators $H_0$ and $H_1$.
In addition, we will choose the same boundary conditions for $H_\eps$ as
for $H_0$ and $H_1$.

\begin{lemma} \label{lemevheps}
Suppose $\phi_0-\phi_1\ge 0$ (resp.\ $\phi_0-\phi_1\leq 0$).
Then the eigenvalues of $H_\eps$ are analytic functions with respect to
$\eps$ and they are decreasing (resp.\ increasing).
\end{lemma}

In particular, this implies that $\dim\Ran P_{(-\infty,\lam)} (H_\eps)$ is continuous
from below (resp. above) in $\eps$ if $\phi_0-\phi_1\ge 0$ (resp.\ $\phi_0-\phi_1\leq 0$).

Now the proof of Theorem~\ref{thm:reg} follows literally as in \cite{kt}.

\section{Relative oscillation criteria}
\label{sec:roc}

As in the previous sections, we will consider two Dirac
operators $\tau_j$, $j=0,1$, and corresponding self-adjoint operators $H_j$, $j=0,1$.
Now we want to answer the question, when a boundary point $E$ of the essential
spectrum of $H_0$ is an accumulation point of eigenvalues of $H_1$.
By Theorem~\ref{thm:rosc} we need to investigate if $\tau_1-E$ is relatively oscillatory
with respect to $\tau_0-E$ or not, that is, if the difference of Pr\"ufer angels
$\Delta_{1,0}=\theta_1-\theta_0$ is bounded or not.

Hence the first step is to derive an ordinary differential equation for $\Delta_{1,0}$. While
this can easily be done by subtracting the differential equations for
$\theta_1$ and $\theta_0$, the result turns out to be not very effective for our purpose.
However, since the number of weighted sign flips $\#_{(c,d)}(u_0,u_1)$ is all we
are eventually interested in, any {\em other} Pr\"ufer angle which gives the same
result will be as good:

\begin{definition}\label{def:eqpa}
We will call a continuous function $\psi$ a Pr\"ufer angle for the Wronskian $W(u_0,u_1)$, if 
$\#_{(c,d)}(u_0,u_1) = \ceil{\psi(d) / \pi} - \floor{\psi(c) / \pi} -1$ for any $c,d\in(a,b)$.
\end{definition}

Hence we will try to find a more effective Pr\"ufer angle $\psi$ than $\Delta_{1,0}$ for the Wronskian
of two solutions. The right choice for Sturm--Liouville equations was found by Rofe-Beketov
\cite{rb2} (see also the recent monograph \cite{rbk}) and it turns out the analogous
definition is also the right one for Dirac operators \cite{kms2}: 

Let $u_0, v_0$ be two linearly independent solutions of $(\tau_0-\lam) u = 0$ with $W(u_0,v_0)=1$
and let $u_1$ be a solution of $(\tau_1-\lam) u = 0$. Define $\psi$ via
\be\label{def:psi}
W(u_0,u_1)=  -R \sin(\psi), \qquad W(v_0,u_1)= - R \cos(\psi).
\ee
Since $W(u_0,u_1)$ and $W(v_0,u_1)$ cannot vanish simultaneously, $\psi$ is a well-defined
absolutely continuous function, once one value at some point $x_0$ is fixed.

\begin{lemma}
The function $\psi$ defined in \eqref{def:psi} is a Pr\"ufer angle for the Wronskian $W(u_0,u_1)$.
\end{lemma}

\begin{proof}
Since $W(u_0,u_1)=  -R \sin(\psi) = -\rho_{u_0}\rho_{u_1}\sin(\Delta_{1,0})$ it suffices
to show that $\psi = \Delta_{1,0} \mod 2\pi$ at each zero of the Wronskian. Since we can assume
$\theta_{v_0}-\theta_{u_0}\in(0,\pi)$ (by $W(u_0,v_0)=1$), this follows by
comparing signs of $R\cos(\psi)= \rho_{v_0}\rho_{u_1}\sin(\theta_{u_1}-\theta_{v_0})$.
\end{proof}

\begin{lemma}\label{lem:lemusualprue}
Let $u_0, v_0$ be two linearly independent solutions of $(\tau_0-\lam) u = 0$ with $W(u_0,v_0)=1$
and let $u_1$ be a solution of $(\tau_1-\lam) u = 0$.

Then the Pr\"ufer angle $\psi$ for the Wronskian $W(u_0,u_1)$ defined in \eqref{def:psi}
obeys the differential equation
\be\label{eq:diff1wronski}
\psi' = - \spr{u_0 \cos(\psi) - v_0 \sin(\psi)}{\Delta \phi (u_0 \cos(\psi) - v_0 \sin(\psi))},
\ee
where
\[
\Delta \phi = \phi_1 -\phi_0.
\]
\end{lemma}

\begin{proof}
Observe $R \psi' = -W(u_0,u_1)' \cos(\psi) + W(v_0,u_1)' \sin(\psi)$
and use \eqref{dwr}, \eqref{def:psi} to evaluate the right-hand side.
\end{proof}

To proceed we will need the following formula for a second solution of a Dirac equation which can
be verified by a straightforward calculation:

\begin{lemma}[{\cite{rb6}, \cite[Lem. 1]{kms2}}]
Let $u$ be a nontrivial solution of $\tau u =  z u$ and choose $x_0 \in I$. Then
\begin{equation}\label{alembert}
v(x)= \left( 2 \int_{x_0}^x \frac{\spr{u(r)}{\hat\phi(r) u(r)}}{|u(r)|^4} dr
- \I \frac{\sig_2}{|u(x)|^2}\right) u(x),
\end{equation}
where
\be
\hat\phi(x) = (m + \phi_{\rm sc}(x)) \sig_3 + \phi_{\rm am}(x) \sig_1,
\ee
is a second linearly independent solution satisfying $W(u,v)=1$.
\end{lemma}

Now we will choose $v_0$ to be given by \eqref{alembert} and, following Schmidt \cite{kms2}, perform a Kepler transformation
\be
\cot(\vphi(x)) = \frac{1}{x} \left( \cot(\psi(x)) -2 \int_{a}^x \frac{\spr{u_0(r)}{\hat\phi_0(r) u_0(r)}}{|u(r)|^4} dr  \right)
\ee
to obtain
\begin{align} \nn
\vphi'(x) =& \frac{1}{x} \bigg( 2 \frac{\spr{u_0(x)}{\hat\phi_0(x) u_0(x)}}{|u_0(x)|^4} \sin^2(\varphi(x)) +
\sin(\varphi(x)) \cos(\varphi(x)) - \\ \nn
&\Big\langle \Big(\cos(\varphi(x)) - \I \frac{\sin(\varphi(x))}{|u_0(x)|^2} \sig_2 \Big) u_0(x), \\
&x^2\Delta\phi(x)\Big(\cos(\varphi(x)) - \I \frac{\sin(\varphi(x))}{|u_0(x)|^2} \sig_2 \Big) u_0(x)\Big\rangle \bigg).
\end{align}
Here we assume that $a>0$ is regular and $b=\infty$ without loss of generality.
Under the further assumption that $|u_0(x)|$, $|u_0(x)|^{-1}$, and $x^2\Delta\phi(x)$ are bounded this
simplifies to
\be
\vphi'(x) = \frac{1}{x}\Big(A(x) \sin^2(\vphi(x)) + \sin(\vphi(x))\cos(\vphi(x)) + B(x) \cos^2(\vphi(x))\Big) + O(x^{-2}),
\ee
where
\be
A(x) = 2\frac{\spr{u_0(x)}{\hat\phi_0(x) u_0(x)}}{|u_0(x)|^4} \quad \textnormal{and}\quad
B(x) = -\spr{u_0(x)}{x^2 \Delta\phi(x) u_0(x)}.
\ee
Now we turn to the case where $\phi_0(x)$ is periodic with period $\alpha>0$ and choose $u_0$ to be the
(anti-)periodic solution at a band edge. Taking averages
\be
\ol{\vphi}(x) = \frac{1}{\alpha} \int_x^{x+\alpha} \vphi(r) dr
\ee
the above differential equation turns into (see \cite[Section~5]{kt3})
\be
\ol{\vphi}'(x) = \frac{1}{x}\Big(\ol{A} \sin^2(\ol{\vphi}(x)) + \sin(\ol{\vphi}(x))\cos(\ol{\vphi}(x)) + B(x) \cos^2(\ol{\vphi}(x))\Big) + O(x^{-2}),
\ee
where
\begin{align}\nn
\ol{A} &= \frac{2}{\alpha} \int_0^\alpha \frac{\spr{u_0(x)}{\hat\phi_0(x) u_0(x)}}{|u_0(x)|^4} dx,\\
\ol{B}(x) &= - \frac{1}{\alpha} \int_x^{x+\alpha} \spr{u_0(r)}{r^2 \Delta\phi(r) u_0(r)}.
\end{align}
Moreover, if $\phi_1(x)$ is given by \eqref{asumper} then one computes
\be
\ol{B}(x) = -\frac{1}{4} \sum_{k=0}^n \frac{x^2}{L_k(x)^2} \ol{B}_k + o(x^2 L_n(x)^{-2}).
\ee
Now we use the following result:

\begin{lemma}[{\cite[Lemma~4.7]{kt3}}]\label{lem:odebnd}
Fix some $n\in\N_0$, let $Q$ be locally integrable on $(a,\infty)$ and abbreviate
\[
Q_n(x) = -\frac{1}{4} \sum_{j=0}^{n-1} \frac{1}{L_j(x)^2}.
\]
Then all solutions of the differential equation
\be
\vphi'(x) = \frac{1}{x} \left(\sin^2(\vphi(x)) + \sin(\vphi(x)) \cos(\vphi(x)) -
x^2 Q(x) \cos^2(\vphi(x)) \right) + o\Big(\frac{ x}{L_n(x)^2}\Big)
\ee
tend to $\infty$ if 
\[
\limsup_{x\to\infty} L_n(x)^2 \left(Q(x) - Q_n(x)\right) < -\frac{1}{4}
\]
and are bounded from above if 
\[
\liminf_{x\to\infty} L_n(x)^2 \left(Q(x) - Q_n(x)\right) > -\frac{1}{4}.
\]
In the last case all solutions are bounded under the additional assumption
$Q =Q_n(x) + O(L_n(x)^{-2})$.
\end{lemma}

Now this lemma implies Theorem~\ref{thm:periodic} if $\ol{A}=1$. However, if $\ol{A}>0$
we can easily reduce it to the case $\ol{A}=1$ by the simple scaling $u_0(x) \to (\ol{A})^{1/2} u_0(x)$.
which renders $\ol{A} \to 1$ and $\ol{B}_k \to \ol{A} \ol{B}_k$. Similarly, if $\ol{A}<0$ we can
reduce it to the case $\ol{A}>0$ via the transformation $\vphi \to - \vphi$ which renders
$\ol{A} \to - \ol{A}$, $\ol{B}_k \to -\ol{B}_k$. Finally, in the case $\ol{A}=0$ the result follows by
using Proposition~1 from \cite{kms2} (Lemma~5.1 in \cite{kt3}) in place of the above lemma.

\section*{Acknowledgments}

We thank K.M.\ Schmidt for discussions on this topic.
Gerald Teschl gratefully acknowledges the hospitality of the Department of Mathematics of
the University of Missouri--Columbia where parts of this paper were written.

\end{document}